\documentclass[a4paper,oneside,10pt,oldfontcommands]{article}
\usepackage[a4paper, total={426pt, 674pt}]{geometry}

\newcommand{\langue}{anglais}	

\usepackage{ifthen}
\ifthenelse{\equal{\langue}{francais}}{
	\usepackage[french]{babel}}
	{\ifthenelse{\equal{\langue}{anglais}}{\usepackage[english]{babel}}{}}
\usepackage[T1]{fontenc}
\usepackage[utf8]{inputenc}
\usepackage[babel,german=guillemets]{csquotes}
\usepackage{eso-pic}

\usepackage{cite}

\usepackage{lmodern} 
\usepackage{enumitem}
\usepackage{ifthen}
\ifthenelse{\equal{\langue}{francais}}{
	\frenchbsetup{StandardLists=true}}
	{}

\usepackage{graphicx} 
\usepackage{subcaption}

\usepackage{amsmath}
\usepackage{amssymb}
\usepackage{amsthm}
\usepackage{amsfonts}
\usepackage{mathtools}
\usepackage{stmaryrd}
\usepackage{tikz-cd}
\usepackage{multirow}

\ifthenelse{\equal{\langue}{francais}}{
	\newcommand{\theoremenom}{Théorème}
	\newcommand{\propositionnom}{Proposition}
	\newcommand{\lemmenom}{Lemme}
	\newcommand{\corollairenom}{Corollaire}
	\newcommand{\definitionnom}{Définition}
	\newcommand{\remarquenom}{Remarque}
	\newcommand{\exemplenom}{Exemple}
	\newcommand{\conjecturenom}{Conjecture}
}{\ifthenelse{\equal{\langue}{anglais}}{
	\newcommand{\theoremenom}{Theorem}
	\newcommand{\propositionnom}{Proposition}
	\newcommand{\lemmenom}{Lemma}
	\newcommand{\corollairenom}{Corollary}
	\newcommand{\definitionnom}{Definition}
	\newcommand{\remarquenom}{Remark}
	\newcommand{\exemplenom}{Example}
	\newcommand{\conjecturenom}{Conjecture}
}{}}
	
\newtheorem{theoreme}{\theoremenom}
\newtheorem{proposition}[theoreme]{\propositionnom}
\newtheorem{lemme}[theoreme]{\lemmenom}

\newtheorem{definition}[theoreme]{\definitionnom}
\newtheorem{remarque}[theoreme]{\remarquenom}

\usepackage{thmtools}
\usepackage{thm-restate}

\makeatletter
\def\cleartheorem#1{%
    \expandafter\let\csname#1\endcsname\relax
    \expandafter\let\csname c@#1\endcsname\relax
}
\makeatother
\newcommand{\compteurThm}{1}
\newcounter{annexe}

\newcommand{\R}{\mathbb{R}}
\newcommand{\N}{\mathbb{N}}

\usepackage{hyperref}

\begin{document}

\pagestyle{empty} 


\title{Existence of an optimal shape for the first eigenvalue of polyharmonic operators}
\author{
Roméo Leylekian
\footnote{Aix-Marseille Université, CNRS, I2M, Marseille, France - \textbf{email:} romeo.leylekian@univ-amu.fr}
}
\date{} 
\maketitle


\begin{abstract}
We prove the existence of an open set minimizing the first eigenvalue of the Dirichlet polylaplacian of order $m\geq1$ under volume constraint. Moreover, the corresponding eigenfunction is shown to enjoy $C^{m-1,\alpha}$ Hölder regularity. This is performed for dimension $2\leq d\leq 4m$. In particular, our analysis answers the question of the existence of an optimal shape for the clamped plate up to dimension $8$.
\end{abstract}

\noindent
\textbf{Acknowledgements.}
Let me thank Dorin Bucur for suggestions about constructing an open optimal set. Thank you to Michel Pierre and to Antoine Lemenant for interesting discussions and references on capacities of higher order. Lastly, I would like to thank Enea Parini and François Hamel for their help and their useful remarks.



\pagestyle{plain} 


\section{Introduction}

The purpose of this document is to prove the existence and the partial regularity of an optimal shape for the first eigenvalue of Dirichlet polylaplacians with volume constraint. More precisely, for some $m\in\N^*$ (fixed from now on), we consider the following eigenvalue problem:
\begin{equation}\label{eq:equation vp}
\left\{
\begin{array}{rcll}
(-\Delta)^m u & = & \Gamma u & \text{in }\Omega, \\
\partial_n^i u & = & 0 & \text{on }\partial\Omega, \text{ for all }0\leq i<m,
\end{array}
\right.
\end{equation}
where $\Omega$ is an open subset of finite volume of $\R^d$ ($d\in\N^*$ is fixed as well), $\partial_n=\nabla\cdot\vec{n}$ is the outward normal derivative, $\Gamma$ is a real number, and $u:\Omega\to\R$ is nontrivial and belongs to the Sobolev space $H_0^m(\Omega)$. It is known that there exists an unbounded sequence of positive such eigenvalues $\Gamma$. Here, we only focus on the lowest one, that we denote $\Gamma(\Omega)$. Then, we are interested in solving
\begin{equation}\label{eq:pb}
\min\{\Gamma(\Omega):\Omega\subseteq\R^d\text{ open set, }|\Omega|=c\},
\end{equation}
where $|\cdot|$ is the Lebesgue measure, and $c$ is a given positive real number. For $m=1$, the Faber-Krahn inequality \cite{faber,krahn} shows that (\ref{eq:pb}) admits the ball as a solution. For $m=2$, in dimension $d=2,3$, it has been proved in 1995 in \cite{nadirashvili} and \cite{ashbaugh-benguria} that the ball also solves (\ref{eq:pb}). Unfortunately, the method employed does not give the result when $d\geq4$. For $m\geq3$, it seems that nothing is known about the solution of (\ref{eq:pb}) in any dimension (except dimension $1$). In this context, it is a natural question to ask whether there is at least one solution to (\ref{eq:pb}). We will give a positive answer to this question for a certain range of dimensions by proving the following.

\begin{theoreme}\label{thm:existence et regularite}
Let $2\leq d\leq 4m$. Problem (\ref{eq:pb}) admits a solution $\Omega$. Moreover, any corresponding eigenfunction extended by $0$ ouside $\Omega$ is $C^{m-1,\alpha}(\R^d)$ for all\/ $0<\alpha<1$.
\end{theoreme}
\clearpage

\begin{remarque}\label{rmq:existence et regularite}
\begin{enumerate}
\item We do not claim anything on the boundedness of an optimal shape.
\item Even if the eigenfunction is somewhat regular, we do not claim any regularity on the optimal shape more than its openness.
\item The regularity of the eigenfunction is not sharp in general (see \cite[Theorem 6.1]{bucur-mazzoleni-pratelli}).
\item In the case $m=1$, the restriction $d\leq 4m$ can be removed (see also Remark \ref{rmq:regularite fonction propre}).
\item The result also holds when $d=1$, for which the problem is trivial. Indeed, any one-dimensional solution for \eqref{eq:pb} must be an interval (otherwise a competitor could be built from one of its connected components), but all the admissible intervals have equal eigenvalue. 
\end{enumerate}
\end{remarque}

The proof of Theorem \ref{thm:existence et regularite} will be divided into two steps, following the approach of \cite{stollenwerk_buck}. First, we prove the existence of an optimal shape in a wider class than the class of open sets. As we shall see, this part of the proof does not require the restriction $2\leq d\leq 4m$. Then, we show that the generalised optimal shape enjoys some regularity properties allowing to solve (\ref{eq:pb}). Let us begin with the existence of a generalised optimal shape.

\subsection{Existence of an optimal shape}

Several breakthroughs have been performed in the direction of exhibiting an optimal shape for (\ref{eq:pb}) or relaxed versions. We cite \cite{bucur-ashbaugh, stollenwerk_buck}, in which an analogous problem is tackled for the buckling bilaplace eigenvalue problem, using a concentration-compactness technique.
Another interesting result is \cite[Theorem 3.5]{bucur}, which establishes the existence of a \textit{quasi-open} optimal shape for a family of spectral optimization problems similar to (\ref{eq:pb}). However, the class of admissible shapes is restricted to some design region. Moreover, we emphasize that the definition chosen in \cite{bucur} for the Sobolev space $H_0^m(\Omega)$ does not match the classical one, unless $\Omega$ enjoys some regularity (see \cite[p. 109]{bucur} for details). As a consequence, even if the optimal shape would be open, it would not necessarily solve (\ref{eq:pb}).
In the same fashion, a similar result but with a different approach shall be found in \cite{stollenwerk_clamped}. There, the existence of an \textit{open} optimal shape for problem (\ref{eq:pb}) with $m=2$ in dimension $d\geq4$ is evidenced, but still among sets lying in a design region.

As a first step towards a proof of Theorem \ref{thm:existence et regularite}, we would like to provide an optimal shape for (\ref{eq:pb}) in the wider class of \textit{quasi-open sets}. Even if we need preliminary notions to understand it correctly, let us give a precise statement for our first main result.

\begin{theoreme}\label{thm:existence d'une forme optimale}
Let $d\in\N^*$. The following problem admits a solution.
\begin{align}
&\min\{\Gamma(\Omega):\Omega\subseteq\R^d\text{, }\Omega\text{ 1-quasi-open, }|\Omega|=c\} \label{pb:Gamma}
\end{align}
\end{theoreme}

\begin{remarque}
\begin{enumerate}
\item Note that the optimal shape need not to be bounded, we only know that it has bounded measure.
\item At first glance it should be suprising that for $m$ arbitrary, the appropriate class of quasi-open sets for getting existence is the class of\/ $1$-quasi-open sets instead of $m$-quasi-open sets. This will become clear as soon as we will understand how the optimal shape is constructed.
\end{enumerate}
\end{remarque}

The notion of quasi-open sets involved in the statement of Theorem \ref{thm:existence d'une forme optimale} will be investigated in section \ref{sec:preliminaires}. Then, the proof of Theorem \ref{thm:existence d'une forme optimale} is achieved in section \ref{subsec:preuve du theoreme d'existence}. However, without entering into details, we shall briefly describe the strategy of the proof: We follow \cite{bucur-ashbaugh}, and apply the concentration-compactness principle (see paragraph \ref{subsec:Concentration-Compactness principle}) directly to the first eigenfunctions of a minimizing sequence of open sets. We show that only compactness shall occur, hence we are able to extract a converging subsequence of first eigenfunctions. The optimal shape is identified as the set of points where the limit function or one of its derivatives up to order $m-1$ does not vanish. Such a set is only $1$-quasi-open in general, which explains why we needed to widen the class of admissible sets.

\subsection{Regularity of the optimal shape}

After providing an optimal shape among quasi-open sets, the next step in the proof of Theorem \ref{thm:existence et regularite} would be to show that this optimal set is actually open, and to give some hint on its regularity. A wide literature have been produced on this subject, especially in the case $m=1$ (see for instance \cite{bucur-mazzoleni-pratelli,bucur-buttazzo-velichkov} and the references given at the end of section \ref{sec:conclusion}). For the case $m=2$, the work \cite{stollenwerk_clamped}, following \cite{stollenwerk}, gives partial answers, but again with the assumption of a prescribed design region. This assumption was dropped successfully only in the case of the buckling bilaplacian \cite{stollenwerk_buck}. Therefore, as for the existence, the regularity of an optimal shape seems still open in full generality. Indeed, gaining regularity from scratch is a quite difficult task in general. As a consequence, we will not prove that any optimal shape for (\ref{pb:Gamma}) is a regular open set. Alternatively, in the wake of \cite{bucur-mazzoleni-pratelli, stollenwerk_buck}, we will obtain regularity for the first eigenfunction, and from this fact build an open set competing with the optimal quasi-open set of Theorem \ref{thm:existence d'une forme optimale}. More precisely, we will prove the following.

\begin{theoreme}\label{thm:regularite fonction propre}
Assume that $2\leq d\leq 4m$ and that there exists a measurable set\/ $\Omega$ solving
\begin{equation}\label{eq:pb mesurable}
\min\{\Gamma(\omega): \omega\subset \R^d\text{ measurable, }|\omega|=c\}.
\end{equation}
Then, any first eigenfunction over $\Omega$ extended by $0$ outside $\Omega$ is $C^{m-1,\alpha}(\R^d)$ for all $0<\alpha<1$. As a consequence, there exists an(other) optimal set $\tilde{\Omega}$ for (\ref{eq:pb mesurable}) which is open.
\end{theoreme}

\begin{remarque}\label{rmq:regularite fonction propre}
\begin{enumerate}
\item The regularity which is obtained for the eigenfunction is not sharp in general (see \cite[Theorem 6.1]{bucur-mazzoleni-pratelli}).
\item The restriction $d\leq 4m$ comes from the fact that Theorem \ref{thm:regularite quasi-min} is used with $f=\Gamma(\Omega) u$ which needs to belong to $L_{loc}^{p}(\R^d)$ for some $p\geq d/m$ (here $u$ is an eigenfunction over $\Omega$). When $d\leq 4m$, this holds by Sobolev embeddings. In dimension greater than $4m$, such an integrability of $u$ is known only when $\Omega$ is already very regular (for instance $C^{2m}$ by elliptic regularity \cite[Theorem~2.20]{gazzola-grunau-sweers} and bootstrap arguments). In the special case $m=1$, this condition on $u$ is always automatically fulfilled since the eigenfunctions are bounded \cite[Example 2.1.8]{davies}. For $m>1$, it would be of interest to study the boundedness of the eigenfunctions in measurable sets or even in open sets with regularity lower than $C^{2m}$.
\end{enumerate}
\end{remarque}

The proof of Theorem \ref{thm:regularite fonction propre} is given in section \ref{subsec:preuve du theoreme de regularite}. The strategy for proving $C^{m-1,\alpha}$ regularity for an eigenfunction $u$ over $\Omega$ is to resort to Morrey's Theorem (see \cite[Chapter III, Theorem~1.1]{giaquinta}) which requires some control on the growth of the $L^2$ norm of the derivatives of $u$ on small balls. As shown in the proof of Theorem \ref{thm:regularite quasi-min}, such a control is precisely a crucial property enjoyed by quasi-minimizers of Dirichlet energy functionals (see paragraph \ref{subsec:quasi-min} for the definition and first properties of quasi-minimizers). As a consequence, one boils down to showing that $u$ is a quasi-minimizer of some Dirichlet energy functional, which will come from the optimality of $\Omega$. After proving the regularity of the eigenfunction, we will define the competitor $\tilde{\Omega}$ as the set of points where $u$ or one of its derivatives up to order $m-1$ does not vanish.

\subsection{Structure of the paper}

The document is structured as follows. Section \ref{sec:preliminaires} contains preliminaries. In particular the notions of capacities, quasi-open sets, and related concepts are recalled. Then, existence results are addressed in section \ref{sec:existence d'une forme optimale}. The method of concentration-compactness adapted to our situation is presented in paragraph \ref{subsec:Concentration-Compactness principle}, after what the proof of Theorem \ref{thm:existence d'une forme optimale} is achieved in paragraph \ref{subsec:preuve du theoreme d'existence}. Section \ref{sec:regularite} deals with regularity results. Local quasi-minimizers are introduced and briefly studied in paragraph \ref{subsec:quasi-min} and the proof of Theorem \ref{thm:regularite fonction propre} is finally performed in paragraph \ref{subsec:preuve du theoreme de regularite}. Section \ref{sec:conclusion} concludes, providing the proof of Theorem \ref{thm:existence et regularite} and discussing some perspectives.

\section{Preliminaries}\label{sec:preliminaires}

We recall that if $\Omega$ is an open set of bounded measure, $H_0^m(\Omega)$ is defined as the closure of $C_c^\infty(\Omega)$ in the Sobolev space $H^m(\Omega)$, and the eigenvalue $\Gamma(\Omega)$ is then characterized variationally by the formula
\begin{equation}\label{eq:Gamma}
\Gamma(\Omega)=\min_{\substack{v\in H_0^m(\Omega)\\v\neq0}}\frac{\int_\Omega|D^m v|^2}{\int_\Omega v^2}.
\end{equation}
Here, $D^m v$ denotes the differential of order $m$ of $v$. Its Euclidean norm $|D^mv|^2=D^mv:D^mv$ is inherited from the Euclidean scalar product denoted
$$
D^m v:D^mw=\sum_{i\in\{1,...,d\}^m}\partial_iv\partial_iw=\sum_{\substack{\gamma\in\N^d, \\ |\gamma|=m}}^d\frac{m!}{\gamma !}\partial^\gamma v\partial^\gamma w,
$$
where for $i=(i_1,...,i_m)\in\{1,...,d\}^m$ we write $\partial_i v=\partial_{i_1}...\partial_{i_m}v$, whereas for $\gamma=(\gamma_1,...,\gamma_d)\in\N^d$ we write $\partial^\gamma v=\partial_1^{\gamma_1}...\partial_d^{\gamma_d}v$ and $\gamma!=\gamma_1!...\gamma_d!$.
It can be shown, thanks to the compact embedding $H_0^m(\Omega)\hookrightarrow L^2(\Omega)$, that the minimum of (\ref{eq:Gamma}) is achieved for some function $u$. Then, $u$ is a first eigenfunction, that is $(u,\Gamma(\Omega))$ is a distributional solution to (\ref{eq:equation vp}). This last point is a consequence of the integration by parts formula,
\begin{equation}\label{eq:ipp}
\int_\Omega \left((-\Delta)^mu\right)v=\int_\Omega D^m u:D^mv,\qquad \forall u,v\in C_c^\infty(\Omega).
\end{equation}

As explained in the introduction, we will need to define the first eigenvalue on sets which are not necessarily open. To this end, we will use the right-hand side of formula (\ref{eq:Gamma}) as a definition, hence we need to investigate the functional space $H_0^m(\Omega)$ in the case where $\Omega$ is a general set. This involves the notions of capacities of orders $1$ up to $m$. For the sake of brevity, we will use the prefix $k$ when talking about capacity of order $k$ and related notions.

We refer to \cite[section 3.3]{henrot-pierre} for the definition and properties of $1$-capacity and $1$-quasi-open sets. We recall that a $1$-quasi-continuous function is a function whose preimage of any open set is a $1$-quasi-open set, and that any function in $H^1(\R^d)$ has a $1$-quasi-continuous representative uniquely defined $1$-quasi everywhere ($1$-q.e. in short, which means up to sets of null $1$-capacity). In the following, we shall always identify $H^1$ functions with their $1$-quasi-continuous representative. Then, Theorem 3.3.42 of \cite{henrot-pierre} shows that, when $\Omega$ is open,
\begin{equation}\label{eq:definition de H_0^1}
H_0^1(\Omega)=\{u\in H^1(\R^d):u=0\text{ 1-q.e. in }\Omega^c\}.
\end{equation}
In the same fashion, one may define an $m$-capacity and the corresponding notions of $m$-quasi-open sets and $m$-quasi-continuous functions (see \cite{adams-hedberg} for a very general theory of capacities). Formally, for $m\in\N$ and a set $E\subseteq\R^d$, the $m$-capacity of $E$, denoted $cap_m(E)$, is given by
$$
cap_m(E):=\inf\{\|v\|_{H^m(\R^d)}^2:\text{ $v\in H^m(\R^d)$, $v\geq1$ a.e. on a neighbourhood of $E$}\}.
$$
Actually, this is not the traditional definition, but an equivalent characterization (see \cite[Proposition 3.3.5]{henrot-pierre}). For our purpose, it is enough to think of the $m$-capacity as an outer measure over $\R^d$ which is finer than the $(m-1)$-capacity in the sense that $cap_{m}(E)=0\Rightarrow cap_{m-1}(E)=0$ for any $E\subseteq\R^d$. Then, $m$-quasi-open sets are defined to be sets which are open up to a residue of arbitrary small $m$-capacity. Lastly, $m$-quasi-continuous functions are functions whose preimages of open sets are $m$-quasi-open sets. In other words,
\begin{definition}
A set $E$ is $m$-quasi-open if for any $\epsilon>0$, there is $\omega\subseteq\R^d$ such that $E\cup\omega$ is open and $cap_m(\omega)<\epsilon$. A function $f:\R^d\to\R$ is $m$-quasi-continuous if for any open set $U\subseteq\R$, $f^{-1}(U)$ is $m$-quasi-open.
\end{definition}
Moreover, any $u\in H^m(\R^d)$ has an $m$-quasi-continuous representative defined $m$-q.e. (that is up to sets of null $m$-capacity), and we will always identify $u$ with its representative. When $\Omega$ is open, the space $H_0^m(\Omega)$ shall be characterised thanks to the capacities up to order $m$ in the following way (see for instance \cite[p.2]{grosjean-lemenant-mougenot}):
\begin{equation}\label{eq:definition de H_0^m}
H_0^m(\Omega)=\{u\in H^m(\R^d):D^k u=0\text{ }(m-k)\text{-q.e. in }\Omega^c\text{, for }0\leq k<m\}.
\end{equation}
Let us point out that in general $H_0^m(\Omega)$ does not coincide with the set of functions in $H^m(\R^d)$ vanishing a.e. outside $\Omega$. Indeed, such a property requires some regularity on $\Omega$ (see \cite{grosjean-lemenant-mougenot}). Observe that as the right-hand member of (\ref{eq:definition de H_0^m}) makes sense regardless of the regularity of $\Omega$, we shall use it to define $H_0^m(\Omega)$ even when $\Omega$ is not open. Moreover, the compact injection
\begin{equation}\label{eq:injection compact}
H_0^m(\Omega)\hookrightarrow L^2(\Omega)
\end{equation}
holding when $\Omega$ was open, remains when $\Omega$ is only a measurable set of finite measure, as explained below.

\begin{definition}
Let $\Omega$ be a subset of $\R^d$. Then, $H_0^m(\Omega)$ is defined by (\ref{eq:definition de H_0^m}).
\end{definition}

\begin{lemme}\label{lemme:injection compacte}
Let $\Omega$ be a measurable set of finite measure. Then, injection (\ref{eq:injection compact}) is compact.
\end{lemme}

\begin{proof}
Let $(u_n)_n$ be a sequence in $H_0^m(\Omega)$. By outer regularity of the Lebesgue measure, there exists an open set $\tilde{\Omega}\supseteq \Omega$ of measure $|\Omega|+1<+\infty$. Then,
$$
H_0^m(\Omega)\subseteq H_0^m(\tilde{\Omega})\hookrightarrow L^2(\tilde{\Omega})
$$
is compact, so we shall extract from $(u_n)_n$ a subsequence converging in $L^2(\tilde{\Omega})$, hence in $L^2(\Omega)$.
\end{proof}

Lastly, as a consequence of Lemma \ref{lemme:injection compacte}, formula (\ref{eq:Gamma}) generalises to any measurable set of finite measure, and the minimum in (\ref{eq:Gamma}) is also achieved in this setting. Before going further, let us mention the following useful observation, which follows immediately from (\ref{eq:definition de H_0^m}).
 
 \begin{lemme}\label{lemme:construction d'un quasi-ouvert}
 Let $u\in H^m(\R^d)$. Then, $u\in H_0^m(\bigcup_{k=0}^{m-1}\{D^k u\neq 0\})$.
 \end{lemme}
 
 \begin{remarque}
 Although Lemma \ref{lemme:construction d'un quasi-ouvert} looks obvious in view of the previous definition of $H_0^m$, it has to be noted that the set $\omega:=\bigcup_{k=0}^{m-1}\{D^k u\neq 0\}$ depends upon the choice of a quasi-continuous representative not only for $u$ but also for its derivatives. Consequently, one may find a set $\tilde{\omega}$ coinciding with $\omega$ only $1$-q.e. and such that $u\in H_0^m(\tilde{\omega})$. This might appear paradoxical at first glance, since the assertion \enquote{$u\in H_0^m(\tilde{\omega})$ for all $\tilde{\omega}=\omega$ $1$-q.e.} is false when $m>1$. The nuance is that we only claim the existence of \textbf{some} $\tilde{\omega}=\omega$ $1$-q.e. such that $u\in H_0^m(\tilde{\omega})$. In other words, $\omega$ is defined more than just $1$-quasi everywhere.
 \end{remarque}

To conclude this section, let us explain the relevance of our procedure of relaxing the class of admissible domains for problem (\ref{eq:pb}). This is justified by the fact that minimizing $\Gamma$ among open sets, quasi-open sets, or even measurable sets does not change the value of the infimum, as stated in the next lemma. Its proof relies on the homogeneity of $\Gamma$, which reads, for any set $\Omega$,

\begin{equation}\label{eq:homogénéité}
\Gamma(\tau\Omega)=\tau^{-2m}\Gamma(\Omega),\qquad \forall \tau>0.
\end{equation}

\begin{lemme}\label{lemme:inf egaux}
For any $k\in\N$,
$$
\inf_{\substack{\Omega\text{ open}\\|\Omega|=c}}\Gamma(\Omega)=\inf_{\substack{\Omega\text{ k-quasi-open}\\|\Omega|=c}}\Gamma(\Omega)=\inf_{\substack{\Omega\text{ measurable}\\|\Omega|=c}}\Gamma(\Omega)>0.
$$
\end{lemme}

\begin{proof}
Clearly,
$$
\inf_{\substack{\Omega\text{ open}\\|\Omega|=c}}\Gamma(\Omega)\geq\inf_{\substack{\Omega\text{ k-quasi-open}\\|\Omega|=c}}\Gamma(\Omega)\geq\inf_{\substack{\Omega\text{ measurable}\\|\Omega|=c}}\Gamma(\Omega),
$$
since any open set is also quasi-open, and any quasi-open set is measurable. For the converse inequality, observe that if $\Omega$ is a measurable set of volume $c$, by outer regularity of the Lebesgue measure there exists for every $\epsilon>0$ an open set $\Omega_\epsilon$ of volume $c+\epsilon$ such that $\Omega\subseteq\Omega_\epsilon$, and hence $H_0^m(\Omega)\subseteq H_0^m(\Omega_\epsilon)$. Then, by (\ref{eq:homogénéité}),
$$
\Gamma(\Omega)\geq\Gamma(\Omega_\epsilon)\geq\left(\frac{c}{c+\epsilon}\right)^{\frac{2m}{d}}\inf_{\substack{\Omega\text{ open}\\|\Omega|=c}}\Gamma(\Omega),
$$
and we get the result by taking $\epsilon\to0$. In order to prove that the infimum does not vanish, one may consider an open set $\Omega$ and combine the classical Faber-Krahn inequality with the following one:
\begin{equation}\label{eq:borne inf}
\lambda(\Omega)^{m}\leq\Gamma(\Omega),
\end{equation}
where $\lambda(\Omega)$ is the first eigenvalue of the Laplacian, which is defined by taking $m=1$ in (\ref{eq:Gamma}). To prove (\ref{eq:borne inf}), we let $u\in H_0^m(\Omega)$ be an eigenfunction (of the polylaplacian of order $m$) normalised in $L^2(\Omega)$. After an integration by parts and two Cauchy-Schwarz inequalities, we observe that for any $1\leq k<m$,
$$
\left(\int_\Omega|D^k u|^2\right)^2\leq\int_\Omega|D^{k-1}u|^2\int_\Omega|D^{k+1}u|^2.
$$
In other words, denoting $v_k=\int_\Omega|D^k u|^2$, we have
$$
v_k^2\leq v_{k-1}v_{k+1}.
$$
As a consequence, we prove by recursion that $(v_k)_k$ obeys the rule
$$
v_{m-k}^{p_k}\leq v_{m-k-1}^{q_k} v_m,\qquad \forall 1\leq k<m,
$$
where $(p_k)_k$ and $(q_k)_k$ are defined by $(p_1,q_1)=(2,1)$, and for all $k\geq1$,
$$
p_{k+1}=2p_k-q_k,\qquad q_{k+1}=p_k.
$$
The couple $(p_k,q_k)_{k\geq1}$ forms a constant-recursive sequence of order $1$, the solution of which is $(p_k,q_k)=(k+1,k)$. Consequently,
\begin{equation}\label{eq:poincaré}
v_1^{m}\leq v_0^{m-1}v_m,
\end{equation}
which gives the result since $u\in H_0^1(\Omega)$ is normalised in $L^2$.
\end{proof}

\section{Existence of an optimal shape for \texorpdfstring{$\Gamma$}{Gamma}}\label{sec:existence d'une forme optimale}

\subsection{Concentration-Compactness principle}\label{subsec:Concentration-Compactness principle}

As mentioned in the introduction, the concentration-compactness principle is the main tool used for proving Theorem \ref{thm:existence d'une forme optimale}. In this paragraph, we adapt to our needs the standard concentration-compactness Lemma of Lions (see Lemma I.1 and III.1 of \cite{lions}). We will use the notation $\tau_y u$ to designate the translation of a given function $u\in H^m(\R^d)$ by a given point $y\in\R^d$, that is the function defined by $\tau_y u(x)=u(x-y)$.

\begin{lemme}\label{lemme:lions}
Let $(u_n)_n$ be a bounded sequence of functions in $H^m(\R^d)$ such that $\|u_n\|_{L^2}^2\to\lambda>0$. Then, up to extracting a subsequence, one of the following occurs:
\begin{enumerate}[label=(\roman*)]
\item (compactness) there exists a sequence of points $(y_n)_n\subseteq\R^d$ such that
$$
\forall \epsilon>0,\exists R>0,\quad\liminf_{n\to\infty}\|\tau_{y_n}u_n\|_{L^2(B_R)}^2\geq\lambda-\epsilon.
$$
\item (vanishing)
$$
\forall R>0,\quad \lim_{n\to\infty}\sup_{y\in\R^d}\|\tau_yu_n\|_{L^2(B_R)}^2=0.
$$
\item (dichotomy) there exists $\alpha\in(0,\lambda)$ and functions $u_n^1$, and $u_n^2$ in $H^m(\R^d)$ such that
\begin{align*}
&\|u_n-(u_n^1+u_n^2)\|_{L^2}\to 0, \\
&\|u_n^1\|_{L^2}^2\to\alpha,\qquad \|u_n^2\|_{L^2}^2\to \lambda-\alpha, \\
&\liminf\left[\|D^m u_n\|_{L^2}^2-\|D^m u_n^1\|_{L^2}^2-\|D^m u_n^2\|_{L^2}^2\right]\geq 0.\\
\intertext{
Moreover, for any $0\leq k<m$, up to some sets of zero $(m-k)$-capacity,
}
&dist\left(\omega_{n,k}^{1},\omega_{n,k}^{2}\right)\to\infty,\\
&\omega_{n,k}^{1},\omega_{n,k}^{2}\subseteq \omega_{n,k},
\end{align*}
where $\omega_{n,k}:=\bigcup_{j=0}^k\{D^j u_n\neq0\}$ and $\omega_{n,k}^{i}:=\bigcup_{j=0}^k\{D^j u_n^i\neq0\}$ for $i=1,2$.
\end{enumerate}
\end{lemme}

\begin{proof}
We follow the proof of Lemma III.1 of \cite{lions} and define
$$
Q_n(R):=\sup_{y\in\R^d}\|\tau_y u_n\|_{L^2(B_R)}^2
$$
which induces a sequence of nonnegative nondecreasing uniformly bounded functions on $\R_+$ satisfying, for each $n\in\N$, $Q_n(R)\xrightarrow[R\to\infty]{}\lambda_n$, with $\lambda_n\to\lambda$. Then, by Helly's selection Theorem, up to extracting a subsequence, $Q_n$ pointwise converges to some function $Q$ on $\R_+$. Then $Q$, like $Q_n$, enjoys the properties of being nonnegative and nondecreasing. Furthermore $Q$ is bounded by $\lambda$. Now set
$$
\alpha := \lim_{R\to\infty}\uparrow Q(R)\in[0,\lambda].
$$
If $\alpha=0$, we are in the vanishing case, because for any $R$, $0\leq Q(R)\leq\alpha=0$, hence $Q(R)=0$ which is the result. If $\alpha=\lambda$, we are in case of compactness. Indeed, let $\epsilon>0$ and choose $r$ such that $\lambda-\epsilon<Q(r)$. Because $Q_n(r)\to Q(r)$, for all $n$ large enough, $\lambda-\epsilon<Q_n(r)$. Then, there exists $z_n\in\R^d$ such that $\lambda-\epsilon<\|\tau_{z_n} u_n\|_{L^2(B_r)}^2$. Note that the radius $r=r(\epsilon)$ and the translation points $z_n=z_n(\epsilon)$ depend on $\epsilon$. It remains to show that, up to increasing $r$, the points $z_n$ shall be replaced by points $y_n$ which are independent of $\epsilon$.

But as long as $\epsilon<\lambda/3$, the balls $B(z_n(\epsilon),r(\epsilon))$ and $B(z_n(\lambda/3),r(\lambda/3))$ intersect each other for all $n$ large enough, otherwise we would have, up to a subsequence,
$$
\|u_n\|_{L^2}^2\geq\|\tau_{z_n(\epsilon)} u_n\|_{L^2(B_{r(\epsilon)})}^2+\|\tau_{z_n(\lambda/3)} u_n\|_{L^2(B_{r(\lambda/3)})}^2>\lambda-\epsilon+\lambda-\lambda/3>4\lambda/3,
$$
which is impossible since $\|u_n\|_{L^2}^2\to\lambda$.
As a consequence, the distance between $z_n(\epsilon)$ and $z_n(\lambda/3)$ does not exceed $r(\epsilon)+r(\lambda/3)$. Then, setting $R(\epsilon):=r(\lambda/3)+2r(\epsilon)$, we get that $B(z_n(\epsilon),r(\epsilon))$ is included in $B(z_n(\lambda/3),R(\epsilon))$. Therefore, the sequence $y_n=z_n(\lambda/3)$ satisfies that for any $0<\epsilon<\lambda/3$, there exists $R:=R(\epsilon)$ such that for all $n$ large enough,
$$
\|\tau_{y_n} u_n\|_{L^2(B_R)}^2\geq\|\tau_{z_n(\epsilon)}u_n\|_{L^2(B_{r(\epsilon)})}^2>\lambda-\epsilon.
$$
This concludes the case of compactness. If now $\alpha\in(0,\lambda)$, let $\epsilon_n\to0$ and $R_n\to\infty$ such that $\forall n\in\N$, $\alpha\geq Q(R_n)>\alpha-\epsilon_n$. Because $Q_n\to Q$ once again, we may extract a subsequence $(u_n)_n$ and define a sequence $(y_n)_n$ such that, \mbox{$\forall n\in\N$, $\alpha-\epsilon_n<\|\tau_{y_n} u_n\|_{L^2(B_{R_n})}^2$}. In the next lines, we will build $u_n^1$ in a way that $\tau_{y_n}u_n^1=\tau_{y_n}u_n$ over $B_{R_n}$ and with support in $B_{2R_n}$. We will also build $u_n^2$ in a way that $\tau_{y_n}u_n^2$ is supported in the complement of $B_{3R_n}$ and equals $\tau_{y_n}u_n$ in the complement of $B_{6R_n}$. In this fashion, the supports of $\tau_{y_n}u_n^1$ and $\tau_{y_n}u_n^1$ will be separated by an annulus of width $R_n\to\infty$. Furthermore, in order to analyse the convergence of $u_n^1$ and $u_n^2$, we will need that $\|\tau_{y_n} u_n\|_{L^2(B_{2R_n})}^2<\alpha+\epsilon_n$ and also $\|\tau_{y_n} u_n\|_{L^2(B_{6R_n})}^2<\alpha+\epsilon_n$, which can be assumed without loss of generality since $Q_n\to Q\leq\alpha$.

Now, let us actually construct $u_n^1$ and $u_n^2$. We use a smooth cut-off function over $B_1$ and with support in $B_2$, say $\chi^1$, and set $\chi^2:=1-\chi^1$. Moreover, we define the $m$-quasi-continuous functions
\begin{alignat*}{2}
u_n^1=u_n\tau_{-y_n}\delta_{R_n}\chi^1, \qquad u_n^2=u_n\tau_{-y_n}\delta_{3R_n}\chi^2,
\end{alignat*}
where $\delta_r v$ is the dilation of a function $v$ by $r$, i.e. $\delta_rv(x):=v(x/r)$. Note that $\delta_r\chi^1$ and $\delta_{3r} \chi^2$, as well as their derivatives, are respectively supported in $B_{2r}$ and in $\R^d\setminus B_{3r}$. As a consequence, we have
$$
dist\left(\omega_{n,k}^1,\omega_{n,k}^2\right)\geq R_n\to\infty.
$$
Moreover, for $i=1,2$, the inclusion $
\omega_{n,k}^i\subseteq\omega_{n,k}$ follows from the definition of $u_n^1$ and $u_n^2$.
Now let us check that
\begin{align*}
& \left\|D^m u_n^1\right\|_{L^2}^2 -\left\|(\tau_{-y_n}\delta_{R_n}\chi^1)D^m u_n\right\|_{L^2}^2 \underset{n\to\infty}{=}O(1/R_n), \\
& \left\|D^m u_n^2\right\|_{L^2}^2 -\left\|(\tau_{-y_n}\delta_{3R_n}\chi^2)D^m u_n\right\|_{L^2}^2 \underset{n\to\infty}{=}O(1/R_n).
\end{align*}
For that purpose, take $\gamma\in\N^d$ a multi-index of length $m$, and recall that
$$
\partial^\gamma(fg)=\sum_{\beta\leq\gamma}\binom{\gamma}{\beta}\partial^\beta f\partial^{\gamma-\beta}g,
$$
from which we deduce
\begin{align*}
\left\|\partial^\gamma u_n^1\right\|_{L^2}^2 & =\left\|(\tau_{-y_n}\chi_{R_n}^1)\partial^\gamma u_n\right\|_{L^2}^2
 + \sum_{\substack{\beta\leq\gamma \\ \beta\neq0}}\binom{\gamma}{\beta}^2\frac{1}{R_n^{2|\beta|}}\left\|(\partial^{\gamma-\beta}u_n)(\tau_{-y_n}\delta_{R_n}\partial^\beta\chi^1)\right\|_{L^2}^2 \\
& + \sum_{\substack{\beta,\beta'\leq\gamma \\ \beta\neq\beta'}}\binom{\gamma}{\beta}\binom{\gamma}{\beta'}\frac{1}{R_n^{|\beta|+|\beta'|}}\left\langle(\partial^{\gamma-\beta}u_n)(\tau_{-y_n}\delta_{R_n}\partial^\beta\chi^1)\text{ , }(\partial^{\gamma-\beta'}u_n)(\tau_{-y_n}\delta_{R_n}\partial^{\beta'}\chi^1)\right\rangle_{L^2}.
\end{align*}
Since $(u_n)_n$ is bounded in $H^m(\R^d)$, and since $\chi^1\in W^{m,\infty}(\R^d)$, the two sums above are $O(1/R_n)$, and we obtain as expected
$$
\left\|\partial^\gamma u_n^1\right\|_{L^2}^2\underset{n\to\infty}{=}\left\|(\tau_{-y_n}\chi_{R_n}^1)\partial^\gamma u_n\right\|_{L^2}^2+O(1/R_n).
$$
An analogous discussion leads to an analogous result regarding $\left\|\partial^\gamma u_n^2\right\|_{L^2}^2$.
Now, using also the inequality $\left|\delta_r\chi^1\right|^2+\left|\delta_{3r}\chi^2\right|^2\leq1$, the previous justifies that for all $n\in\N$,
\begin{align*}
\|D^m u_n\|_{L^2}^2- & \|D^m u_n^1\|_{L^2}^2 - \|D^m u_n^2\|_{L^2}^2 \\
& \geq\|D^m u_n\|_{L^2}^2-\|(\tau_{-y_n}\delta_{R_n}\chi^1)D^m u_n\|_{L^2}^2-\|(\tau_{-y_n}\delta_{3R_n}\chi^2)D^m u_n\|_{L^2}^2 +O(1/R_n) \\
& \geq\|D^m u_n\|_{L^2}^2-\int_{\R^d} \tau_{-y_n}\left(\left|\delta_{R_n}\chi^1\right|^2+\left|\delta_{3R_n}\chi^2\right|^2\right)|D^m u_n|^2+O(1/R_n)\\
& \geq O(1/R_n).
\end{align*}
This means that $\liminf_n\left[\|D^m u_n\|_{L^2}^2-\|D^m u_n^1\|_{L^2}^2-\|D^m u_n^2\|_{L^2}^2\right]\geq 0$.
Moreover,
$$
\|u_n^1\|_{L^2}^2=\|\delta_{R_n}\chi^1\tau_{y_n} u_n\|_{L^2}^2\geq\|\tau_{y_n} u_n\|_{L^2(B_{R_n})}^2>\alpha-\epsilon_n\to\alpha.
$$
On the other hand, since $\delta_{R_n}\chi^1$ is supported in $B_{2R_n}$,
$$
\|u_n^1\|_{L^2}^2=\|\delta_{R_n}\chi^1\tau_{y_n} u_n\|_{L^2}^2\leq\|\tau_{y_n} u_n\|_{L^2(B_{2R_n})}^2<\alpha+\epsilon_n\to\alpha.
$$
Hence $\|u_n^1\|_{L^2}^2\to\alpha$. In the same spirit, regarding $u_n^2$,
$$
\|u_n^2\|_{L^2}^2=\|\delta_{3R_n}\chi^2\tau_{y_n} u_n\|_{L^2}^2\leq\|u_n\|_{L^2}^2-\|\tau_{y_n} u_n\|_{L^2(B_{3R_n})}^2\leq\|u_n\|_{L^2}^2-\alpha+\epsilon_n\to\lambda-\alpha,
$$
and
$$
\|u_n^2\|_{L^2}^2\geq \|u_n\|_{L^2}^2-\|\tau_{y_n}u_n\|_{L^2(B_{6R_n})}^2>\|u_n\|_{L^2}^2-\alpha-\epsilon_n\to\lambda-\alpha,
$$
showing that $\|u_n^2\|_{L^2}^2\to\lambda-\alpha$.
To conclude,
\begin{align*}
\|u_n-(u_n^1+u_n^2)\|_{L^2}^2 & =\int_{\R^d} \left(\delta_{3R_n}\chi^1-\delta_{R_n}\chi^1\right)^2|\tau_{y_n}u_n|^2 \\
& \leq \|\tau_{y_n}u_n\|_{L^2(B_{6R_n})}^2-\|\tau_{y_n}u_n\|_{L^2(B_{R_n})}^2\leq \alpha+\epsilon_n-\alpha+\epsilon_n\to 0.
\end{align*}
\end{proof}

\subsection{Proof of Theorem \ref{thm:existence d'une forme optimale}}\label{subsec:preuve du theoreme d'existence}

Equipped with the concentration-compactness principle, we can now turn to the proof of Theorem~\ref{thm:existence d'une forme optimale}, which claims the existence of a $1$-quasi-open set minimizing $\Gamma$ under volume constraint. Following \cite{bucur-ashbaugh}, the strategy is to apply the concentration-compactness procedure to the eigenfunctions of a minimizing sequence of sets, and to show that only compactness occurs. As a consequence, the eigenfunctions are shown to converge in $L^2$ towards a limit function. The final step is to identify the optimal set from the limit function, which is performed in the next proposition.

\begin{proposition}\label{prop:compacité}
Let $\Omega_n$ be a sequence of measurable sets of volume $c$, and $u_n\in H_0^m(\Omega_n)$ the corresponding eigenfunctions converging to some $u$ in $H^m(\R^d)$ weakly and in $L^2(\R^d)$ strongly. Then there exists a $1$-quasi-open set $\Omega$ of volume $|\Omega|\leq c$ such that
$$
\liminf_n\Gamma(\Omega_n)\geq\Gamma(\Omega).
$$ 
\end{proposition}

To prove Proposition \ref{prop:compacité}, one builds the limit shape $\Omega$ as the set of points where $u$ or one of its derivatives up to order $m-1$ does not vanish. The conclusion is then straightforward, except the fact that $|\Omega|\leq c$. This actually follows from a pointwise convergence of the functions and from Egorov's Theorem, as shown in the lemma below.

\begin{lemme}\label{lemme:convergence pp et support}
Let $N\in\N^*$ and $U_n:\R^d\to\R^N$ form a sequence of measurable functions converging a.e. toward $U:\R^d\to\R^N$. Then \mbox{$|\{U\neq 0\}|\leq\liminf|\{U_n\neq 0\}|$}.
\end{lemme}

\begin{proof}
We assume without loss of generality that $c:=\liminf|\{U_n\neq0\}|$ is finite, and proceed by contradiction, pretending that $|\{U\neq0\}|>c$. Since
$$
\{U\neq0\}=\bigcup_{k\in\N^*}\{|U|\geq 1/k\},
$$
there exists $k\in\N^*$ such that $|\{|U|\geq 1/k\}|>c$. In the same fashion, there exists a large enough ball $B$ such that $|\{|U|\geq 1/k\}\cap B|>c$. By Egorov's Theorem, $B$ being of finite volume, there exists $\omega\subseteq B$ such that $|\{|U|\geq 1/k\}\cap \omega|>c$ and $U_n\to U$ uniformly over $\omega$. As a result for all large enough $n$, $|U_n|>0$ over $\{|U|\geq 1/k\}\cap \omega$, which is of volume greater than $c$. In other words, $\liminf|\{U_n\neq 0\}|>c$, a contradiction.
\end{proof}

\begin{proof}[Proof of Proposition \ref{prop:compacité}]
Thanks to weak $H^m$ and strong $L^2$ convergence of $(u_n)_n$, we immediately have
$$
\liminf_n\Gamma(\Omega_n)\geq\frac{\int_{\R^d}|D^m u|^2}{\int_{\R^d}u^2}.
$$
That's why one would like to define the optimal shape as the domain where $u$ or one of its derivatives does not vanish. Fortunately, the set \mbox{$\Omega:=\cup_{k=0}^{m-1}\{D^ku\neq0\}$} is $1$-quasi-open, and \mbox{$u\in H_0^m(\Omega)$} by Lemma \ref{lemme:construction d'un quasi-ouvert}. Therefore,
$$
\liminf_n\Gamma(\Omega_n)\geq\Gamma(\Omega).
$$
It remains only to show that $|\Omega|\leq c$. For that purpose, observe that not only $(u_n)_n$ but also $(D^k u_n)_n$ converges strongly in $L^2(\R^d)$ for all $0\leq k<m$, by interpolation (recall Gagliardo-Nirenberg inequalities). Moreover, up to a subsequence, we shall assume that the convergence holds almost everywhere. Applying Lemma \ref{lemme:convergence pp et support} to the functions $U_n:=(u_n,\nabla u_n,\dots,D^{m-1} u_n)$, we find that $|\Omega|\leq\liminf|\{U_n\neq 0\}|$, and thanks to the $1$-q.e. inclusion $\{U_n\neq 0\}\subseteq\Omega_n$, this shows as desired that $|\Omega|\leq c$.
\end{proof}

{\noindent
Let us now apply the concentration-compactness principle and proceed to the proof of Theorem~\ref{thm:existence d'une forme optimale}.
}

\begin{proof}[Proof of Theorem \ref{thm:existence d'une forme optimale}]
Consider a minimizing sequence $\Omega_n$ for problem (\ref{pb:Gamma}). Actually, thanks to Lemma \ref{lemme:inf egaux}, we can take the same minimizing sequence for problems (\ref{pb:Gamma}) and (\ref{eq:pb}), and hence choose $\Omega_n$ to be open. Moreover, we have $\inf_n\Gamma(\Omega_n)=:I>0$. We note $u_n\in H_0^m(\Omega_n)$ the eigenfunction on $\Omega_n$, and assume without loss of generality that $u_n$ is normalised in $L^2$. Then, by Lemma \ref{lemme:lions}, we extract a subsequence from $(u_n)_n$ such that three situations shall occur: compactness, vanishing, or dichotomy.

If we are in the vanishing situation, for any $0<R<\infty$,
$$
\lim_n\sup_{y\in\R^d}\int_{B_R}|\tau_yu_n|^2 = 0.
$$
In this case, if $(u_n)_n$ were bounded in $H^1$, then thanks to \cite[Lemma 3.6]{bucur-ashbaugh} there would exist a sequence of points $(y_n)_n\subseteq\R^d$ such that $\tau_{y_n}u_n$ admit no subsequence weakly converging to $0$ in $H^1$. Consequently, there would exist $\epsilon>0$ and a test function $\varphi$ with support in a ball of radius say $R>0$, such that for $n$ large enough,
$$
\epsilon\leq\left(\int_{\R^d}\tau_{y_n}u_n\varphi\right)^2\leq\sup_{y\in\R^d}\int_{B_R}|\tau_yu_n|^2\int_{B_R}\varphi^2.
$$
This would contradict the fact that $(u_n)_n$ vanishes. Therefore, we can extract a subsequence such that $\int_{\Omega_n}|\nabla u_n|^2\to\infty$. But thanks to inequality (\ref{eq:poincaré}), we obtain that $\int_{\Omega_n}|D^m u_n|^2\to\infty$, showing in turn that $\Gamma(\Omega_n)\to\infty$. To conclude, the vanishing situation cannot occur.

If the dichotomy occurs, there exists $\alpha\in (0,1)$ and sequences $(u_n^1)_n$, $(u_n^2)_n$ in $H^m(\R^d)$ with
\begin{align*}
&\|u_n-(u_n^1+u_n^2)\|_{L^2}\to 0, \\
&\|u_n^1\|_{L^2}^2\to\alpha,\qquad \|u_n^2\|_{L^2}^2\to \lambda-\alpha, \\
&\liminf\left[\|D^m u_n\|_{L^2}^2-\|D^m u_n^1\|_{L^2}^2-\|D^m u_n^2\|_{L^2}^2\right]\geq 0,\\
&dist\left(\omega_{n,k}^1,\omega_{n,k}^2\right)\to\infty, & \text{up to sets of null }(m-k)\text{-capacity},\\
&\omega_{n,k}^1,\omega_{n,k}^2\subseteq \omega_{n,k}\subseteq\Omega_n, & \text{up to sets of null }(m-k)\text{-capacity},
\end{align*}
where $\omega_{n,k}:=\bigcup_{j=0}^k\{D^j u_n\neq0\}$ and $\omega_{n,k}^{i}:=\bigcup_{j=0}^k\{D^j u_n^i\neq0\}$ for $i=1,2$. Then, thanks to the disjunction of $\{u_n^1\neq0\}$ and $\{u_n^2\neq0\}$, we have that for any $\epsilon>0$, for any $n$ large enough,
\begin{align*}
&\int_{\Omega_n} |u_n|^2\leq\int_{\Omega_n}|u_n^1|^2+\int_{\Omega_n}|u_n^2|^2+2\epsilon,
&\int_{\Omega_n} |D^m u_n|^2\geq\int_{\Omega_n}|D^m u_n^1|^2+\int_{\Omega_n}|D^m u_n^2|^2-2\epsilon.
\end{align*}
Therefore,
$$
\frac{\int_{\Omega_n} |D^m u_n|^2}{\int_{\Omega_n} u_n^2}\geq\frac{\int_{\Omega_n}|D^m u_n^1|^2+\int_{\Omega_n}|D^m u_n^2|^2-2\epsilon}{\int_{\Omega_n}|u_n^1|^2+\int_{\Omega_n}|u_n^2|^2+2\epsilon}.
$$
Note that $u_n^1,u_n^2\in H_0^m(\Omega_n)$, hence we have,
$$
\int_{\Omega_n} |u_n^i|^2\leq\frac{1}{\Gamma(\Omega_n)}\int_{\Omega_n} |D^mu_n^i|^2\leq I^{-1}\int_{\Omega_n} |D^mu_n^i|^2,\qquad i=1,2.
$$
As a consequence, both $\liminf_n\int_{\Omega_n}|D^mu_n^1|^2>0$ and $\liminf_n\int_{\Omega_n}|D^mu_n^2|^2>0$, so we may choose $\epsilon$ small enough for having $\int_{\Omega_n}|D^mu_n^i|^2-\epsilon>0$ for $i=1,2$ and any large $n$.  Then, using the algebraic property $\frac{a+b}{c+d}\geq\min\left(\frac{a}{c},\frac{b}{d}\right)$ for $a,b,c,d>0$, and passing to the limit $n\to\infty$, we obtain that
$$
\liminf_n\frac{\int_{\Omega_n} |D^m u_n|^2}{\int_{\Omega_n} u_n^2}\geq\min\left(\liminf_n\frac{\int_{\Omega_n} |D^m u_n^1|^2-\epsilon}{\int_{\Omega_n} |u_n^1|^2+\epsilon},\liminf_n\frac{\int_{\Omega_n} |D^m u_n^2|^2-\epsilon}{\int_{\Omega_n} |u_n^2|^2+\epsilon}\right).
$$
Shrinking $\epsilon\to 0$ and assuming that the minimum is achieved for the quotient involving $u_n^1$, we get
$$
\liminf_n\Gamma(\Omega_n)\geq\liminf_n\frac{\int_{\Omega_n} |D^m u_n^1|^2}{\int_{\Omega_n} |u_n^1|^2}.
$$
Now, for $i=1,2$, define $\omega_n^i:=\omega_{n,m-1}^i=\bigcup_{0\leq k<m}\{D^k u_n^i\neq0\}$. Observe that $\omega_n^i$ is $1$-quasi-open, and that, by Lemma \ref{lemme:construction d'un quasi-ouvert}, $u_n^i\in H_0^m(\omega_n^i)$ for $i=1,2$. In particular,
$$
\liminf_n\Gamma(\Omega_n)\geq\liminf_n\Gamma(\omega_n^1).
$$
Moreover, since $\omega_n^1$ and $\omega_n^2$ are $1$-q.e. disjoint (recall that $dist(\omega_n^1,\omega_n^2)\to\infty$) subsets of $\Omega_n$, we have
$$
|\omega_n^1|\leq|\Omega_n|-|\omega_n^2|.
$$
But $\liminf_n|\omega_n^2|> 0$. Indeed, since $u_n^2$ belongs to $H_0^m(\omega_n^2)$, we have thanks to relation (\ref{eq:homogénéité}) and to Lemma \ref{lemme:inf egaux},
$$
|\omega_n^2|^{\frac{2m}{d}}\frac{\int \left|D^m u_n^2\right|^2}{\int \left|u_n^2\right|^2}\geq|\omega_n^2|^{\frac{2m}{d}}\Gamma(\omega_n^2)\geq c^{\frac{2m}{d}}\inf\{\Gamma(\Omega):\Omega\text{ 1-quasi-open, }|\Omega|=c\}=c^{\frac{2m}{d}} I>0.
$$
Now because $u_n^2$ is uniformly bounded in $H^m(\R^d)$, and because $\|u_n^2\|_{L^2}^2\to1-\alpha\neq0$, we end up with $$
\liminf_n|\omega_n^2|> 0.
$$
As a result, we have $\limsup_n|\omega_n^1|=:c'<c$. Consequently, $\Omega_n':=\left(c/|\omega_n^1|\right)^{\frac{1}{d}}\omega_n^1$ is an admissible sequence made of $1$-quasi-open sets of volume $c$, hence the following contradiction
$$
\liminf_n\Gamma(\omega_n^1)\leq\liminf_n\Gamma(\Omega_n)\leq\liminf_n\Gamma(\Omega_n')=\left(c'/c\right)^{\frac{2m}{d}}\liminf_n\Gamma(\omega_n^1).
$$
We have shown that neither vanishing nor dichotomy can occur. The only remaining possibility is compactness. This means that there exists a sequence of points $(y_n)_n\subset\R^d$ such that for all $\epsilon>0$, there exists some radius $R>0$ such that, for any $n$ large enough,
$$
\|\tau_{y_n}u_n\|_{L^2(B_R)}\geq1-\epsilon.
$$
First, for readability, considering $\Omega_n+y_n$ instead of $\Omega_n$, it is possible to assume that $y_n=0$. Now, note that $(u_n)_n$ is bounded in $L^2(\R^d)$, hence converges weakly to some $u\in L^2(\R^d)$, up to a subsequence. We will see that the convergence is also $L^2$ strong. Indeed, let $(R_k)_k$ be an increasing sequence of radiuses such that, for each $k\in\N^*$, for each large $n\in\N,\|u_n\|_{L^2(B_{R_k})}\geq1-1/k$.
Note that for each $k\in\N^*$, the sequence $(u_n)_n$ is bounded in $H^m(B_{R_k})$, hence we may extract diagonaly a subsequence converging strongly in each $L^2(B_{R_k})$ to $u$. As a consequence, $\|u\|_{L^2}\geq 1$ since for any $k\in\N^*$, \mbox{$\|u\|_{L^2}\geq\|u\|_{L^2(B_{R_k})}=\lim_n\|u_n\|_{L^2(B_{R_k})}\geq 1-1/k$}. On the other hand, by weak convergence $\|u\|_{L^2}\leq1$. Therefore, in $L^2(\R^d)$, the convergence holds simultaneously weakly and in norm, hence strongly.

Weak $H^m$ and strong $L^2$ convergence of $(u_n)_n$ towards $u$, together with Proposition \ref{prop:compacité} show that it is possible to build a $1$-quasi-open set $\Omega$, of volume lower or equal than $c$ such that $\liminf_n\Gamma(\Omega_n)\geq\Gamma(\Omega)$. Therefore, $\Omega$, or more precisely its dilation fitting the volume constraint, solves (\ref{pb:Gamma}).
\end{proof}

\section{Regularity of an optimal shape}\label{sec:regularite}

\subsection{Quasi-minimizers}\label{subsec:quasi-min}

In optimization, a common strategy for getting regularity is to resort to the notion of local quasi-minimizers\footnote{We mention that as far as we know, there is no direct link between capacity and quasi-minimizers. In particular, the prefix \enquote{quasi} of quasi-minimizers does not refer to any capacity.} as defined in \cite{bucur-mazzoleni-pratelli}, based on the ideas of \cite{briancon-hayouni-pierre,alt-caffarelli}. In this paragraph, we recall the notion of local quasi-minimizers of a functional defined on $H^m(\R^d)$.

\begin{definition}\label{def:quasi-min}
Let $J:H^m(\R^d)\to\R$ be a functional and $u\in H^m(\R^d)$. We say that $u$ is a local quasi-minimizer for $J$ if there exists $C,\alpha,r_0>0$ such that for all $x_0\in\R^d$ and $0< r<r_0$,
\begin{equation}\label{eq:def quasi-min}
J(u)\leq J(v)+Cr^d,
\end{equation}
for all $v\in\mathcal{A}_{r,\alpha}:=\{v\in H^m(\R^d):u-v\in H_0^m(B_r(x_0)), \int_{B_r(x_0)}|D^m (u-v)|^2\leq\alpha\}$.
\end{definition}

\begin{remarque}
The constants $C$, $\alpha$, and $r_0$ may depend not only on $J$, but also on $u$.
\end{remarque}

For the archetypal case of Dirichlet energy functionals, a central property of quasi-minimizers is given in the next lemma.

\begin{lemme}\label{lemme:quasi-minimalité}
Let $f\in L^2(\R^d)$ and $J_f:H^m(\R^d)\to\R$ defined by
$$
J_f(v)=\frac{1}{2}\int_{\R^d}|D^mv|^2-\int_{\R^d}fv.
$$
Let $u\in H^m(\R^d)$ be a local quasi-minimizer for $J_f$. Then, there exists $C,r_0>0$ (depending upon $d$,$m$,$f$ and $u$) such that for all $x_0\in\R^d$ and $0< r<r_0$,
\begin{equation}\label{eq:propriété des quasi-min}
\langle (-\Delta)^mu-f,\varphi\rangle_{H^{-m},H_0^m}\leq Cr^{d/2}\left(\int_{B_r(x_0)}|D^m\varphi|^2\right)^{1/2}, \qquad \forall \varphi\in H_0^m(B_r(x_0)).
\end{equation}
\end{lemme}

\begin{remarque}
In the setting of Lemma \ref{lemme:quasi-minimalité}, (\ref{eq:propriété des quasi-min}) is not only implied but is actually equivalent to (\ref{eq:def quasi-min}) (see Remark 3.3 and Remark 3.4 of \cite{bucur-mazzoleni-pratelli}).
\end{remarque}

\begin{proof}
Let $\psi\in H_0^m(B_r(x_0))$, such that $\int_{B_r(x_0)}|D^m \psi|^2\leq\alpha$. Because $\psi\in H_0^m(B_r(x_0))$, we may integrate by parts (recall (\ref{eq:ipp})) to obtain
$$
\langle (-\Delta)^mu-f,\psi\rangle_{H^{-m},H_0^m}=\int_{\R^d}D^mu:D^m\psi-\int_{\R^d}f\psi.
$$
On the other hand, $v:=u-\psi\in \mathcal{A}_{r,\alpha}$, hence from the quasi-minimality of $u$ we get
$$
J_f(u)\leq J_f(v)+Cr^d=J_f(u)+J_f(-\psi)-\int_{\R^d}D^mu:D^m\psi +Cr^d,
$$
from which we deduce
$$
\int_{\R^d}D^mu:D^m\psi-\int_{\R^d}f\psi\leq \frac{1}{2}\int_{\R^d}|D^m\psi|^2+Cr^d.
$$
Now if $\varphi\in H_0^m(B_r(x_0))$, set $\psi=\alpha^{1/2} \left(\frac{r}{r_0}\right)^{d/2}\|D^m \varphi\|_{L^2(B_r(x_0))}^{-1}\varphi$ so that $\psi\in H_0^m(B_r(x_0))$ and $\int_{B_r(x_0)}|D^m \psi|^2\leq\alpha$. The previous discussion shows that
$$
\alpha^{1/2} \left(\frac{r}{r_0}\right)^{d/2}\|D^m \varphi\|_{L^2(B_r(x_0))}^{-1}\langle (-\Delta)^mu-f,\varphi\rangle_{H^{-m},H_0^m}\leq\frac{\alpha}{2}\left(\frac{r}{r_0}\right)^{d}+Cr^d.
$$
In other words, there exists $C'>0$ such that for all $r<r_0$,
$$
\langle (-\Delta)^mu-f,\varphi\rangle_{H^{-m},H_0^m}\leq C'r^{d/2}\left(\int_{B_r(x_0)}|D^m\varphi|^2\right)^{1/2},
$$
which concludes.
\end{proof}

\subsection{Proof of Theorem \ref{thm:regularite fonction propre}}\label{subsec:preuve du theoreme de regularite}

To prove Theorem \ref{thm:regularite fonction propre}, the first step is to gain regularity on an optimal eigenfunction. For that purpose, the concept of local quasi-minimizer provides an approriate framework. Indeed, the following result gives the regularity of local quasi-minimizers of Dirichlet energy functionals. The regularity obtained is not sharp in general, as one shall see in \cite[Theorem 3.6]{bucur-mazzoleni-pratelli}.

\begin{theoreme}\label{thm:regularite quasi-min}
Assume $d\geq2$. Let $f\in L^2\cap L_{loc}^p(\R^d)$, $p\geq\frac{d}{m}$ and $J_f:H^m(\R^d)\to\R$ defined by
\begin{equation}\label{eq:fonctionnelle energetique}
J_f(v)=\frac{1}{2}\int_{\R^d}|D^mv|^2-\int_{\R^d}fv.
\end{equation}
Let $u\in H^m(\R^d)$ be a local quasi-minimizer for $J_f$. Then, $u\in C^{m-1,\alpha}(\R^d)$ for all $0<\alpha<1$.
\end{theoreme}

\begin{proof}
We proceed in a similar way than \cite[Theorem 3 to Theorem 4]{stollenwerk}. More precisely, our goal is to prove that for all $\alpha\in(0,1)$, there exists $R_0,C>0$ (depending on $d$,$m$,$p$,$f$,$u$, and $\alpha$) such that, for all $x_0\in\R^d$ and $0<r<R_0$,
\begin{equation}\label{eq:morrey}
\int_{B_r(x_0)}|D^m u|^2\leq Cr^{d-2+2\alpha}.
\end{equation}
Indeed, due to Morrey's Theorem (see \cite[Chapter III, Theorem 1.1]{giaquinta}), this yields that $D^{m-1}u\in C^{0,\alpha}$ for all $\alpha\in(0,1)$. In order to show (\ref{eq:morrey}), we let $x_0\in\R^d$, $R>0$ and define $v_{x_0,R}$ such that $u-v_{x_0,R}\in H_0^m(B_R(x_0))$ and $(-\Delta)^mv_{x_0,R}=f$ in $B_R(x_0)$. Such a function may be constructed by minimizing $J_f$ over the set $u-H_0^m(B_R(x_0))$. Now we will estimate the two terms in the right-hand side of
\begin{equation}\label{eq:D^mu sur Br}
\int_{B_r(x_0)}|D^m u|^2\leq 2\int_{B_r(x_0)}|D^m (u-v_{x_0,R})|^2+2\int_{B_r(x_0)}|D^m v_{x_0,R}|^2.
\end{equation}
We will control the first term in the larger ball $B_R(x_0)$ proceeding as in the proof of \cite[inequality (A.8)]{bucur-mazzoleni-pratelli}. Indeed, integrating by parts, and then using the definition of $v_{x_0,R}$, we get
\begin{align}\label{eq:intégrale D^m(u-v) sur B_R}
\int_{B_R(x_0)}|D^m (u-v_{x_0,R})|^2 & =\langle (-\Delta)^m (u-v_{x_0,R}),(u-v_{x_0,R})\rangle_{H^{-m},H_0^m} \nonumber\\
& =\langle (-\Delta)^m u-f,(u-v_{x_0,R})\rangle_{H^{-m},H_0^m}.
\end{align}
The quasi-minimality of $u$ and Lemma \ref{lemme:quasi-minimalité} show that there exists $C,R_0>0$ such that, as long as $R<R_0$,
\begin{equation}\label{eq:quasi-min de u}
\langle (-\Delta)^m u-f,(u-v_{x_0,R})\rangle_{H^{-m},H_0^m}\leq CR^{d/2}\left(\int_{B_R(x_0)}|D^m (u-v_{x_0,R})|^2\right)^{1/2}.
\end{equation}
Finally, combining (\ref{eq:intégrale D^m(u-v) sur B_R}) and (\ref{eq:quasi-min de u}), we obtain (still for all $R<R_0$),
\begin{equation}\label{eq:D^m u-v}
\int_{B_R(x_0)}|D^m (u-v_{x_0,R})|^2\leq CR^{d}.
\end{equation}
For the second term in (\ref{eq:D^mu sur Br}), let us decompose $v_{x_0,R}$ into $y_{x_0,R}$ and $z_{x_0,R}$ , which are defined in the following way: $y_{x_0,R}$ is the function in $H_0^m(B_R(x_0))$ such that $(-\Delta)^my_{x_0,R}=f$; and $z_{x_0,R}$ is the $m$-polyharmonic function such that $v_{x_0,R}-z_{x_0,R}\in H_0^m(B_R(x_0))$. Notice that, using the definition of $y_{x_0,R}$, the Cauchy-Schwarz inequality, and relation (\ref{eq:homogénéité}), we find
\begin{equation*}
\begin{split}
\int_{B_R(x_0)}|D^my_{x_0,R}|^2=\int_{B_R(x_0)} fy_{x_0,R} & \leq \|f\|_{L^2(B_R(x_0))}\|y_{x_0,R}\|_{L^2(B_R(x_0))} \\
& \leq CR^{m}\|f\|_{L^2(B_R(x_0))}\|D^my_{x_0,R}\|_{L^2(B_R(x_0))}.
\end{split}
\end{equation*}
In particular, using Hölder's inequality (with $q=p/2$) and the fact that $f\in L^p(B_{R_0}(x_0))$,
\begin{equation}\label{eq:D^m y}
\int_{B_R(x_0)}|D^my_{x_0,R}|^2\leq CR^{2m}\int_{B_R(x_0)}f^2\leq CR^{2m+d(1-2/p)}\leq CR^d,
\end{equation}
the last inequality coming from $m-d/p\geq 0$ and $R\leq R_0$. Observe that Hölder's inequality cannot be applied when $p\leq 2$, but in this case $d\leq 2m$, hence (\ref{eq:D^m y}) follows from $f\in L^2(\R^d)$. On the other hand, using inequality (3.2) of \cite{giaquinta-modica} for polyharmonic functions (see also \cite[Lemma 2]{stollenwerk}), we obtain the existence of $C>0$ such that, for all $0<r<R<R_0$,
$$
\int_{B_r(x_0)}|D^mz_{x_0,R}|^2\leq C\left(\frac{r}{R}\right)^d\int_{B_R(x_0)}|D^mz_{x_0,R}|^2.
$$
Because $z_{x_0,R}$ minimizes $J_0$ amongst functions $z$ such that $v_{x_0,R}-z\in H_0^m(B_R(x_0))$, the previous turns into the next
\begin{equation}\label{eq:D^m z}
\int_{B_r(x_0)}|D^mz_{x_0,R}|^2\leq C\left(\frac{r}{R}\right)^d\int_{B_R(x_0)}|D^mu|^2.
\end{equation}
Combining (\ref{eq:D^m y}) and (\ref{eq:D^m z}), we find some $C>0$ such that for all $0<r<R<R_0$,
\begin{equation}\label{eq:D^m v}
\int_{B_r(x_0)}|D^mv_{x_0,R}|^2\leq C\left(R^d+\left(\frac{r}{R}\right)^d\int_{B_R(x_0)}|D^mu|^2\right).
\end{equation}
Thanks to (\ref{eq:D^m u-v}) and (\ref{eq:D^m v}), we can now conclude that for some $C>0$ and for all $0<r<R<R_0$,
\begin{equation}
\int_{B_r(x_0)}|D^m u|^2\leq C\left[R^d+\left(\frac{r}{R}\right)^d\int_{B_R(x_0)}|D^mu|^2\right].
\end{equation}
In view of \cite[Chapter III, Lemma 2.1, p.86]{giaquinta} applied to $\phi(r)=\int_{B_r(x_0)}|D^m u|^2$ with $\alpha=d$ and $\beta=d-\gamma$ for any $0<\gamma<d$, we conclude that, for all $0<r<R<R_0$,
$$
\int_{B_r(x_0)}|D^m u|^2\leq C\left(\frac{r}{R}\right)^{d-\gamma}\left[\int_{B_{R}(x_0)}|D^mu|^2+R^{d-\gamma}\right],
$$
where $C$ depends as usual on $d$,$m$,$p$,$f$,$u$, but now also on $\gamma$. Fixing for instance $R=R_0/2$, we end up with
$$
\int_{B_r(x_0)}|D^m u|^2\leq Cr^{d-\gamma},
$$
for all $r<R_0/4$, which is exactly (\ref{eq:morrey}) as long as we set $\gamma=2-2\alpha<2\leq d$ where $0<\alpha<1$.
\end{proof}

Thanks to Theorem \ref{thm:regularite quasi-min}, it remains only to show that an optimal eigenfunction is a local quasi-minimizer, which is performed below and allows to conclude the proof of Theorem \ref{thm:regularite fonction propre}.

\begin{proof}[Proof of Theorem \ref{thm:regularite fonction propre}]\label{preuve:regularite fonction propre}
To begin with, we explain the last part of Theorem \ref{thm:regularite fonction propre}, which asserts that, as long as an eigenfunction $u$ associated with $\Gamma(\Omega)$ is $C^{m-1,\alpha}(\R^d)$, we shall construct an open set $\tilde{\Omega}$ which is optimal. Indeed, we proceed in a similar fashion to the proof of \cite[Theorem~6.1]{bucur-mazzoleni-pratelli}, and define
$$
\tilde{\Omega}:=\bigcup_{0\leq k<m}\{D^k u\neq0\}.
$$
Then, $\tilde{\Omega}$ is open and, by Lemma \ref{lemme:construction d'un quasi-ouvert}, $u\in H_0^m(\Omega)$. In particular, $\Gamma(\tilde{\Omega})\leq\Gamma(\Omega)$. It remains to check that $\tilde{\Omega}$ satisfies the volume constraint. But since $u\in H_0^m(\Omega)$, we know that, for each $k<m$, $\{D^k u\neq0\}$ is included in $\Omega$ up to a set of null $(m-k)$-capacity, and hence $\tilde{\Omega}\subseteq \Omega$ up to a set of null $1$-capacity. As a consequence, $|\tilde{\Omega}|\leq|\Omega|$, and $\tilde{\Omega}$ satisfies the volume constraint up to a dilation, which concludes.

Let us now turn to the proof of the fact that $u$ is $C^{m-1,\alpha}(\R^d)$. In view of Theorem \ref{thm:regularite quasi-min}, it is enough to show that $u$ is a local quasi-minimizer of the functional $J_f$ defined in (\ref{eq:fonctionnelle energetique}) for some $f\in L^2\cap L^{p}(\R^d)$ with $p\geq d/m$. For that purpose, we introduce the concept of quasi-minimizer of a shape functional: we say that $\Omega$ is a local quasi-minimizer of the functional $\Gamma$ if there exists $C,r_0>0$ such that, for all $x_0\in\R^d$ and $0<r<r_0$, and for any measurable set $\Omega'$,
$$
\Omega'\Delta \Omega\subseteq B_r(x_0)\quad\Longrightarrow \quad\Gamma(\Omega)\leq \Gamma(\Omega')+Cr^d,
$$
where $\Omega'\Delta \Omega$ is the symmetric difference between $\Omega$ and $\Omega'$. As shown in \cite[Lemma 4.6]{bucur-mazzoleni-pratelli}, the local quasi-minimality of $u$ with respect to $J_f$ shall be deduced from the local quasi-minimality of $\Omega$ with respect to $\Gamma$. Indeed, without loss of generality, we assume $u$ to be $L^2$ normalised. Then, as long as $\Omega$ is a local quasi-minimizer, for any $v\in H^m(\R^d)$ with $v-u\in H_0^m(B_r(x_0))$ and $r\in(0,r_0)$, we have that $v\in H_0^m(\Omega\cup B_r(x_0))$, hence, using the variational definition of $\Gamma(\Omega\cup B_r(x_0))$,
\begin{equation}\label{eq:quasi-min de forme}
\Gamma(\Omega)\leq \Gamma(\Omega\cup B_r(x_0))+Cr^d\leq \frac{\int_{\R^d}|D^mv|^2}{\int_{\R^d}v^2}+Cr^d,
\end{equation}
Now let $\alpha>0$ and take an arbitrary $v\in\mathcal{A}_{r,\alpha}$ (recall Definition \ref{def:quasi-min}). Because $u$ is normalised, $r<r_0$ and $\int_{B_r(x_0)}|D^m(u-v)|^2<\alpha$ we have
$$
\int_{\R^d} v^2\leq2\int_{\R^d} u^2+2\int_{\R^d} (u-v)^2\leq2\int_{\R^d} u^2+\frac{2\alpha}{\Gamma(B_{r_0}(x_0))}\leq C,
$$
where $C$ does not depend on $v$. Hence after multiplying (\ref{eq:quasi-min de forme}) by $\int_{\R^d}v^2$ we obtain
$$
\Gamma(\Omega)\int_{\R^d}v^2\leq\int_{\R^d}|D^mv|^2+Cr^d.
$$
Using the inequality $v^2\geq2uv-u^2$, we find
$$
-\Gamma(\Omega)\int_{\R^d}u^2\leq \int_{\R^d}|D^mv|^2-2\Gamma(\Omega)\int_{\R^d}uv+ Cr^d.
$$
Observe that this reads $J_f(u)\leq J_f(v)+Cr^d$ for $f=\Gamma(\Omega)u$. As a result, $u$ is a local quasi-minimizer of $J_f$ with $f=\Gamma(\Omega)u$. At this point, we need to investigate the summability of $f$. Remark that thanks to Sobolev injections, $u\in L^q(\R^d)$ for $q=2d/(d-2m)$ if $2m<d$, whereas $u\in L^r(\R^d)$ for all $r\in[2,\infty)$ if $2m\geq d$. The restriction $d\leq 4m$ yields $q\geq d/m$. Consequently, in any case $f\in L^p(\R^d)$ for some $p\geq d/m$, and thanks to Theorem \ref{thm:regularite quasi-min} $u$ is then $C^{m-1,\alpha}(\R^d)$. Thus, it remains only to show the local quasi-minimality of $\Omega$ with respect to $\Gamma$. As mentioned in the second point of Remark 5.2 in \cite{bucur-mazzoleni-pratelli}, this will follow if we show that $\Omega$ is a local shape supersolution for $\Gamma+\Lambda|\cdot|$ for some $\Lambda>0$, or, in other words, that there is some $r_0>0$ such that for all $x_0\in\R^d$ and all measurable set $\Omega\subseteq \Omega'\subseteq \Omega\cup B_{r_0}(x_0)$,
$$
\Gamma(\Omega)+\Lambda|\Omega|\leq\Gamma(\Omega')+\Lambda|\Omega'|.
$$
Indeed, if $\Omega$ is such a shape super-solution, for any $x_0\in\R^d$, any $0<r<r_0$, and any measurable set $\Omega'$ with $\Omega'\Delta \Omega\subseteq B_r(x_0)$, we have that $\Omega\subseteq \Omega'\cup\Omega\subseteq \Omega\cup B_{r_0}(x_0)$, hence
$$
\Gamma(\Omega')\geq\Gamma(\Omega'\cup\Omega)\geq\Gamma(\Omega)+\Lambda|\Omega|-\Lambda|\Omega'\cup\Omega|\geq\Gamma(\Omega)-\Lambda|\Omega'\setminus\Omega|,
$$
which yields as expected
$$
\Gamma(\Omega)\leq\Gamma(\Omega')+\Lambda|B_r(x_0)|.
$$
To conclude, let us show that $\Omega$ is a shape super-solution. This is a consequence of the optimality of $\Omega$, as explained in the proof of \cite[Theorem 6.1]{bucur-mazzoleni-pratelli}, since for any $\Omega\subseteq \Omega'\subseteq \Omega\cup B_{r_0}(x_0)$, setting $t=\left(|\Omega'|/|\Omega|\right)^{\frac{1}{d}}$ and recalling (\ref{eq:homogénéité}),
\begin{align*}
\Gamma(\Omega)\leq\Gamma(\Omega'/t)\leq t^{2m}\Gamma(\Omega') &\leq \Gamma(\Omega')+(t^{2m}-1)\Gamma(\Omega')\\
&\leq \Gamma(\Omega')+C_{r_0}(t^d-1)\Gamma(\Omega)\leq \Gamma(\Omega')+\Lambda(|\Omega'\setminus\Omega|),
\end{align*}
with $\Lambda=C_{r_0}\Gamma(\Omega)/|\Omega|$. Note that, besides L'Hôpital's rule, we used the fact that $t\geq1$ is bounded when $\Omega'\subseteq \Omega\cup B_{r_0}(x_0)$ in order to get the bound $t^{2m}-1\leq C_{r_0}(t^d-1)$.
\end{proof}

\section{Conclusion}\label{sec:conclusion}

Fortified by Theorem \ref{thm:existence d'une forme optimale} and Theorem \ref{thm:regularite fonction propre}, we shall propose a proof for Theorem \ref{thm:existence et regularite}.

\begin{proof}[Proof of Theorem \ref{thm:existence et regularite}]
Let $2\leq d\leq 4m$. Due to Theorem \ref{thm:existence d'une forme optimale}, there exists a $1$-quasi-open set solving (\ref{pb:Gamma}). From Lemma \ref{lemme:inf egaux} we deduce that this set also solves (\ref{eq:pb mesurable}). Then, as desired, Theorem \ref{thm:regularite fonction propre} produces an open set $\Omega$ solving (\ref{eq:pb}). Moreover, $\Omega$ also solves (\ref{eq:pb mesurable}) by Lemma \ref{lemme:inf egaux}. Applying again Theorem \ref{thm:regularite fonction propre}, we conclude that any first eigenfunction over $\Omega$ is $C^{m-1,\alpha}(\R^d)$.
\end{proof}

Theorem \ref{thm:existence et regularite} is of course an advance in the context of shape optimization. However, as stated in Remark \ref{rmq:existence et regularite}, it has several limits besides the dimensional restriction. Indeed, in optimization, existence results are often used as a preliminary step in the derivation of optimality conditions in view of identifying an optimizer. Nevertheless, in shape optimization, some regularity is frequently needed on the optimal shape to apply the theory of shape derivatives in order to derive optimality conditions. See \cite[Chapter 5]{henrot-pierre} for a general theory on shape derivatives. See for instance \cite{buoso-lamberti13,buoso-lamberti15,ortega-zuazua,mohr}, \cite[section 5]{antunes-buoso-freitas}, and \cite[section 3]{leylekian} for results on shape derivatives involving fourth and higher order operators. Finally, see \cite{bucur-ashbaugh} for an example where shape derivatives are performed without regularity assumption on the optimal shape.

Unfortunately, as mentioned in Remark \ref{rmq:existence et regularite}, the optimal shape of Theorem \ref{thm:existence et regularite} enjoys only little regularity. A first step towards upgrading it would be to prove boundedness. Some results are known in this direction, but only in the case of second order operators, up to our knowledge (see \cite{bucur2012,mazzoleni}). After boundedness, the question of finer regularity of the optimal set shall be tackled. See \cite{lamboley-pierre} for a review on this subject. The main results known are covered by some works of Lamboley et al. \cite{briancon-lamboley,dephilippis-lamboley-pierre-velichkov,lamboley-prunier}. However, as far as we know, those results hold only in special situations (for instance $m=1$ with \cite{briancon-lamboley}), or when ancillary constraints allow to gain regularity (see \cite{dephilippis-lamboley-pierre-velichkov,lamboley-prunier}). To conclude let us mention that \cite{stollenwerk_clamped} also contains some hints on this topic in the case $m=2$ (see for instance \cite[Lemma 9]{stollenwerk_clamped}).


\appendix




\begin{thebibliography}{BMPV155}
{}
\bibitem[AH96]{adams-hedberg}
Adams, D. R. and Hedberg, L. I. \emph{Function spaces and potential theory}. Vol. 314.
Grundlehren der mathematischen Wissenschaften [Fundamental Principles of Mathematical
Sciences]. Springer-Verlag, Berlin, 1996, pp. xii+366.
{}
\bibitem[AC81]{alt-caffarelli}
Alt, H. W. and Caffarelli, L. A. “Existence and regularity for a minimum problem with free
boundary”. In: \emph{J. Reine Angew. Math.} 325 (1981), pp. 105-144.
{}
\bibitem[ABF19]{antunes-buoso-freitas}
Antunes, P. R. S., Buoso, D., and Freitas, P. “On the behavior of clamped plates under
large compression”. In: \emph{SIAM J. Appl. Math.} 79.5 (2019), pp. 1872-1891.
{}
\bibitem[AB95]{ashbaugh-benguria}
Ashbaugh, M. S. and Benguria, R. D. “On {R}ayleigh’s conjecture for the clamped plate
and its generalization to three dimensions”. In: \emph{Duke Math. J.} 78.1 (1995),
pp. 1-17.
{}
\bibitem[AB03]{bucur-ashbaugh}
Ashbaugh, M. S. and Bucur, D. “On the isoperimetric inequality for the buckling of a
clamped plate”. In: \emph{Z. Angew. Math. Phys.} 54.5 (2003). Special issue dedicated to
Lawrence E. Payne, pp. 756-770.
{}
\bibitem[BHP05]{briancon-hayouni-pierre}
Briançon, T., Hayouni, M., and Pierre, M. “Lipschitz continuity of state functions in some
optimal shaping”. In: \emph{Calc. Var. Partial Differential Equations} 23.1 (2005),
pp. 13–32.
{}
\bibitem[BL09]{briancon-lamboley}
Briançon, T. and Lamboley, J. “Regularity of the optimal shape for the first eigenvalue of
the {L}aplacian with volume and inclusion constraints”. In: \emph{Ann. Inst. H. Poincaré
C Anal. Non Linéaire} 26.4 (2009), pp. 1149–1163.
{}
\bibitem[Buc05]{bucur}
Bucur, D. “How to prove existence in shape optimization”. In: \emph{Control Cybernet.}
34.1 (2005), pp. 103–116.
{}
\bibitem[Buc12]{bucur2012}
Bucur, D. “Minimization of the {$k$}-th eigenvalue of the {D}irichlet {L}aplacian”. In:
\emph{Arch. Ration. Mech. Anal.} 206.3 (2012), pp. 1073–1083.
{}
\bibitem[BBV13]{bucur-buttazzo-velichkov}
Bucur, D., Buttazzo, G., and Velichkov, B. “Spectral optimization problems with internal
constraint”. In: \emph{Ann. Inst. H. Poincaré C Anal. Non Linéaire} 30.3 (2013),
pp. 477–495.
{}
\bibitem[BMPV15]{bucur-mazzoleni-pratelli}
Bucur, D., Mazzoleni, D., Pratelli, A., and Velichkov, B. “Lipschitz regularity of the
eigenfunctions on optimal domains”. In: \emph{Arch. Ration. Mech. Anal.} 216.1 (2015),
pp. 117–151.
{}
\bibitem[BL13]{buoso-lamberti13}
Buoso, D. and Lamberti, P. D. “Eigenvalues of polyharmonic operators on variable
domains”. In: \emph{ESAIM Control Optim. Calc. Var.} 19.4 (2013), pp. 1225–1235.
{}
\bibitem[BL15]{buoso-lamberti15}
Buoso, D. and Lamberti, P. D. “On a classical spectral optimization problem in linear
elasticity”. In: \emph{New trends in shape optimization}. Vol. 166. Internat. Ser. Numer.
Math. Birkhäuser/Springer, Cham, 2015, pp. 43–55.
{}
\bibitem[Dav90]{davies}
Davies, E. B. \emph{Heat kernels and spectral theory}. Vol. 92. Cambridge Tracts in
Mathematics. Cambridge University Press, Cambridge, 1990, pp. x+197.
{}
\bibitem[DLPV18]{dephilippis-lamboley-pierre-velichkov}
De Philippis, G., Lamboley, J., Pierre, M., and Velichkov, B. “Regularity of minimizers of
shape optimization problems involving perimeter”. In: \emph{J. Math. Pures Appl. (9)}
109 (2018), pp. 147–181.
{}
\bibitem[Fab23]{faber}
Faber, G. \emph{Beweis, da{{ß}} unter allen homogenen {Membranen} von gleicher
{Fläche} und gleicher {Spannung} die kreisförmige den tiefsten {Grundton} gibt.}
Sitz. Bayer. Akad. Wiss., 1923, pp. 169-172.
{}
\bibitem[GGS10]{gazzola-grunau-sweers}
Gazzola, F., Grunau, H.-C., and Sweers, G. \emph{Polyharmonic boundary value
problems}. Vol. 1991. Lecture Notes in Mathematics. Positivity preserving and nonlinear
higher order elliptic equations in bounded domains. Springer-Verlag, Berlin, 2010,
pp. xviii+423.
{}
\bibitem[Gia83]{giaquinta}
Giaquinta, M. \emph{Multiple integrals in the calculus of variations and nonlinear elliptic
systems}. Vol. 105. Annals of Mathematics Studies. Princeton University Press, Princeton,
NJ, 1983, pp. vii+297.
{}
\bibitem[GM79]{giaquinta-modica}
Giaquinta, M. and Modica, G. “Regularity results for some classes of higher order nonlinear
elliptic systems”. In: \emph{J. Reine Angew. Math.} 311/312 (1979), pp. 145–169.
{}
\bibitem[GLM23]{grosjean-lemenant-mougenot}
Grosjean, J.-F., Lemenant, A., and Mougenot, R. \emph{Stable domains for higher order
elliptic operators}. Preprint at \url {https://arxiv.org/abs/2307.07217}. 2023.
{}
\bibitem[HP05]{henrot-pierre}
Henrot, A. and Pierre, M. \emph{Variation et optimisation de formes}. Vol. 48.
Mathématiques \& Applications (Berlin). Une analyse géométrique. Springer, Berlin, 2005,
pp. xii+334.
{}
\bibitem[Kra25]{krahn}
Krahn, E. “Über eine von {R}ayleigh formulierte {M}inimaleigenschaft des {K}reises”. In:
\emph{Math. Ann.} 94.1 (1925), pp. 97–100.
{}
\bibitem[LP17]{lamboley-pierre}
Lamboley, J. and Pierre, M. “Regularity of optimal spectral domains”. In: \emph{Shape
optimization and spectral theory}. De Gruyter Open, Warsaw, 2017, pp. 29–77.
{}
\bibitem[LP23]{lamboley-prunier}
Lamboley, J. and Prunier, R. “Regularity in shape optimization under convexity
constraint”. In: \emph{Calc. Var. Partial Differential Equations} 62.3 (2023), Paper No.
101, 44.
{}
\bibitem[Ley23]{leylekian}
Leylekian, R. “Sufficient conditions yielding the {R}ayleigh {C}onjecture for the
clamped plate”. In: \emph{Ann. Mat. Pura Appl. (4)} 203.6 (2024), 2529-2547.
{}
\bibitem[Lio84]{lions}
Lions, P.-L. “The concentration-compactness principle in the calculus of variations. {T}he
locally compact case. {I}”. In: \emph{Ann. Inst. H. Poincaré Anal. Non Linéaire} 1.2
(1984), pp. 109–145.
{}
\bibitem[MP13]{mazzoleni}
Mazzoleni, D. and Pratelli, A. “Existence of minimizers for spectral problems”. In:
\emph{J. Math. Pures Appl. (9)} 100.3 (2013), pp. 433–453.
{}
\bibitem[Moh75]{mohr}
Mohr, E. “Über die {R}ayleighsche {V}ermutung: unter allen {P}latten von gegebener
{F}läche und konstanter {D}ichte und {E}lastizität hat die kreisförmige den tiefsten
{G}rundton”. In: \emph{Ann. Mat. Pura Appl. (4)} 104 (1975), pp. 85–122.
{}
\bibitem[Nad95]{nadirashvili}
Nadirashvili, N. S. “Rayleigh’s conjecture on the principal frequency of the clamped plate”.
In: \emph{Arch. Rational Mech. Anal.} 129.1 (1995), pp. 1–10.
{}
\bibitem[OZ00]{ortega-zuazua}
Ortega, J. H. and Zuazua, E. “Generic simplicity of the spectrum and stabilization for a
plate equation”. In: \emph{SIAM J. Control Optim.} 39.5 (2000), pp. 1585–1614.
{}
\bibitem[Sto21]{stollenwerk_clamped}
Stollenwerk, K. \emph{Existence of an optimal domain for minimizing the fundamental
tone of a clamped plate of prescribed volume in arbitrary dimension}. Preprint at \url
{https://arxiv.org/abs/2109.01455}. 2021.
{}
\bibitem[Sto22]{stollenwerk}
Stollenwerk, K. “Existence of an optimal domain for the buckling load of a clamped plate
with prescribed volume”. In: \emph{Ann. Mat. Pura Appl. (4)} 201.4 (2022),
pp. 1677–1704.
{}
\bibitem[Sto23]{stollenwerk_buck}
Stollenwerk, K. “On the optimal domain for minimizing the buckling load of a clamped
plate”. In: \emph{Z. Angew. Math. Phys.} 74.1 (2023), Paper No. 12, 16.
\end{thebibliography}




\end{document}